\documentclass[preprint]{elsarticle}
\usepackage{mathrsfs}
\usepackage{amssymb}
\usepackage{amsthm}
\biboptions{sort&compress,square,numbers,comma}

\def\newblock{\hskip 0.11em \@plus 0.33em \@minus -0.07em}

\newtheorem{thm}{Theorem}[section]
\newtheorem{lem}[thm]{Lemma}
\newtheorem{cor}[thm]{Corollary}
\newtheorem{prop}[thm]{Proposition}
\newtheorem{conjecture}[thm]{Conjecture}

\newdefinition{definition}[thm]{Definition}

\begin{document}

\begin{frontmatter}

\title{On the existence and number of $(k+1)$-kings in $k$-quasi-transitive digraphs.}

\author[IMATE]{Hortensia Galeana-S\'anchez}
\ead{hgaleana@matem.unam.mx}
\author[IMATE]{C\'esar Hern\'andez-Cruz\corref{cor1}}
\ead{cesar@matem.unam.mx}
\author[FC]{Manuel Alejandro Ju\'arez-Camacho}
\ead{talex@ciencias.unam.mx}

\address[IMATE]{Instituto de Matem\'aticas, Universidad Nacional Aut\'onoma de M\'exico, Ciudad Universitaria, C.P. 04510, M\'exico, D.F., M\'exico}
\address[FC]{Facultad de Ciencias, Universidad Nacional Aut\'onoma de M\'exico, Ciudad Universitaria, C.P. 04510, M\'exico, D.F., M\'exico}

\cortext[cor1]{Corresponding author}

\begin{abstract}
Let $D=(V(D), A(D))$ be a digraph and $k \ge 2$ an integer.   We say that $D$ is $k$-quasi-transitive if for every directed path $(v_0, v_1, \dots, v_k)$ in $D$, then $(v_0, v_k) \in A(D)$ or $(v_k, v_0) \in A(D)$.   Clearly, a $2$-quasi-transitive digraph is a quasi-transitive digraph in the usual sense.

Bang-Jensen and Gutin proved that a quasi-transitive digraph $D$ has a $3$-king if and only if $D$ has a unique initial strong component and, if $D$ has a $3$-king and the unique initial strong component of $D$ has at least three vertices, then $D$ has at least three $3$-kings. In this paper we prove the following generalization: A $k$-quasi-transitive digraph $D$ has a $(k+1)$-king if and only if $D$ has a unique initial strong component, and if $D$ has a $(k+1)$-king then, either all the vertices of the unique initial strong components are $(k+1)$-kings or the number of $(k+1)$-kings in $D$ is at least $(k+2)$.
\end{abstract}




\end{frontmatter}

\section{Introduction}

We will denote by $D$ a finite digraph without loops or multiple arcs in the same direction, with vertex set $V(D)$ and arc set $A(D)$.   All walks, paths and cycles will be considered to be directed.   For undefined concepts and notation we refer the reader to \cite{BJD} and \cite{BM}.

We say that a vertex $u \in V(D)$ dominates a vertex $v \in V(D)$ if $(u,v) \in A(D)$, and denote it by $u \to v$; consequently, $u \not \to v$ will denote that $(u,v) \notin A(D)$.   The \emph{out-neighborhood} $N^+(v)$ of a vertex $v$ is the set $\left\{ u \in V(D) \colon\ v \to u \right\}$.   The \emph{out-degree} $d^+(v)$ of a vertex $v$ is defined as $d^+(v) = | N^+(v) |$.  Definitions of \emph{in-neighborhood} and \emph{in-degree} of a vertex $v$ are analogously given.   The maximum (resp. minimum) out-degree (resp. in-degree) of a vertex in $D$ will be denoted by $\Delta^+_D$ (resp. $\Delta^-_D$).   If $\mathscr{C} = (x_0, x_1, \dots, x_n)$ is a walk and there are integers $i$ and $j$ such that $0 \le i < j \le n$, then $x_i \mathscr{C} x_j$ will denote the subwalk $(x_i, x_{i+1}, \dots, x_{j-1}, x_j)$ of $\mathscr{C}$.   Union of walks will be denoted by concatenation or with $\cup$.

A digraph is \emph{strongly connected} (or \emph{strong}) if for every $u,v \in V(D)$, there exists a $uv$-path.   A \emph{strong component} (or component) of $D$ is a maximal strong subdigraph of $D$.   The \emph{condensation} of $D$ is the digraph $D^\star$ with $V(D^\star)$ equal to the set of all strong components of $D$, and $(S,T) \in A(D^\star)$ if and only if there is an $ST$-arc in $D$.   Clearly $D^\star$ is an acyclic digraph (a digraph without directed cycles), and thus, it has vertices of out-degree equal to zero and vertices of in-degree equal to zero.   A \emph{terminal component} of $D$ is a strong component $T$ of $D$ such that $d_{D^\star}^+ (T) = 0$.  An \emph{initial component} of $D$ is a strong component $S$ of $D$ such that $d_{D^\star}^- (S)=0$.

A \emph{semicomplete} digraph is a digraph $D$ in which, for every pair of vertices $x, y \in V(D)$, $(x,y) \in V(D)$ or $(y,x) \in V(D)$.

For $u,v \in V(D)$, the distance $d_D (u,v)$ from $u$ to $v$ is the length of the shortest $uv$-path, and in the case there is no $uv$-path in $D$, then $d(u,v) = \infty$.   By the definition, $d(v,v) = 0$.   For a positive integer $k$, a vertex $v$ of $D$ is a $k$-\emph{king} if $d(v,u) \le r$ for each $u \in V(D)$; a \emph{king} is a $2$-king.   Kings were first studied by Landau in \cite{L}, where he proved that every tournament has a king.   A direct consequence of this result is the following.

\begin{thm} \label{sc-2k}
Every semi-complete digraph has a $2$-king.
\end{thm}

Since their introduction, $k$-kings in digraphs have been widely studied.   It was proved, independently in \cite{G} and \cite{P-T},  that every $m$-partite tournament without vertices of in-degree zero has a $4$-king.   Also, in \cite{K-T1}, it is proved that an $m$-partite tournament without vertices of in-degree zero has at least four $4$-kings, if $m = 2$ and at least three $4$-kings if $m > 2$.   The result for the bipartite case was improved in \cite{K-T2}: every bipartite tournament without vertices of in-degree zero and without $3$-kings has at least eight $4$-kings.   Other results about $k$-kings in multipartite tournaments and semicomplete multipartite digraphs can be found in \cite{G-Y,PTr,T}.

A digraph $D$ is $k$-\emph{quasi-transitive} if the existence of a $uv$-path of length $k$ implies $u \to v$ or $v \to u$.   A \emph{quasi-transitive digraph} is a $2$-quasi-transitive digraph.   The family of $k$-quasi-transitive digraphs was introduced in \cite{GS-HC}.   Besides the recursive structural characterization of quasi-transitive digraphs given in \cite{BJQT}, strong $3$-quasi-transitive digraphs were characterized in \cite{GS-G-U}; also, the interaction between strong components of a $3$-quasi-transitive digraph is completely described in \cite{WW2}.   These results were used in \cite{GS-HC} to prove that a $3$-quasi-transitive digraph has a $4$-king if and only if it has a unique initial component.

In \cite{BJ-H2}, Bang-Jensen and Huang proved the following result.

\begin{thm} \label{BJHT}
	Let $D$ be a quasi-transitive digraph, then we have:
	\begin{enumerate}
		\item $D$ has a $3$-king if and only if it has a unique initial strong component.
		\item If $D$ has a $3$-king, then the following holds:
			  \begin{enumerate}
				\item Every vertex in $D$ of maximum out-degree is a $3$-king.
				\item If $D$ has no vertex of in-degree zero, then $D$ has at least two $3$-kings.
				\item If the unique initial strong component of $D$ contains at least three vertices, then D has at least three 3-kings.
			  \end{enumerate}
	\end{enumerate}
\end{thm}

Although no characterization is known for $k$-quasi-transitive digraphs with $k \ge 4$, in \cite{GS-HC} some structural results were obtained and used to prove, for instance, that for every even positive integer $k$, a $k$-quasi-transitive digraph has a $(k+1)$-king if and only if it has a unique initial strong component, generalizing the first statement of Theorem \ref{BJHT}.


This work has three main objectives.   First, to complete the generalization mentioned in the previous paragraph, that is, to prove that for an odd integer $k \ge 5$, a $k$-quasi-transitive digraph has a $(k+1)$-king if and only if it has a unique initial component, hence obtaining the following theorem.

\begin{thm} \label{genBJ}
Let $k \ge 2$ be an integer and $D$ a quasi-transitive digraph.   Then $D$ has a $(k+1)$-king if and only if it has a unique initial strong component.
\end{thm}

Second, we will prove a result analogous to (a) of the second statement of Theorem \ref{BJHT}.   In general it is not true that if a $k$-transitive digraph $D$ has a $(k+1)$-king, then every vertex in $D$ of maximum out-degree is a $(k+1)$-king; consider for example the $4$-transitive digraph $D_4$ with vertex set $\{ v_1, v_2, v_3, v_4 \}$ and arc set $\{ (v_1, v_2), (v_2, v_3), (v_2, v_4), (v_3, v_4), (v_4, v_3) \}$.   Clearly, $v_1$ is the only $5$-king of $D_4$, but the vertex of maximum out-degree is $v_2$.   Despite this fact, if we restrict our search to the unique initial component of $D$, we obtain the following results.

\begin{prop} \label{exis-deg-even-prop}
Let $k \ge 2$ be an even integer and $D$ a $k$-quasi-transitive digraph with unique initial strong component $C$.   If $v \in V(C)$ and $\Delta_C^+ \ge d^+ (v) > \Delta_C^+ (v) - k$, then $v$ is a $(k+1)$-king of $D$.
\end{prop}

\begin{prop} \label{exis-deg-odd-prop}
Let $k \ge 3$ be an odd integer and $D$ a $k$-quasi-transitive digraph with unique initial strong component $C$.   If $v \in V(C)$ and $\Delta_C^+ \ge d^+ (v) > \Delta_C^+ (v) - \frac{k-1}{2}$, then $v$ is a $(k+1)$-king of $D$.
\end{prop}

Third, as it is usual in the study of $k$-kings, we will calculate a lower bound for the number of $(k+1)$-kings in a $k$-transitive digraph with a unique initial component.   Also, for a $k$-quasi-transitive digraph we will give sufficient conditions for the existence of: $3$-kings, for even values of $k$ and $4$-kings, for odd values of $k$.   Our main results are contained in the following theorems.

\begin{thm} \label{main-even}
Let $k \ge 4$ be an even integer and $D$ a $k$-quasi-transitive digraph with a unique initial strong component $C$.   Then at least one of the following statements holds:
\begin{enumerate}
	\item Every vertex in $C$ is a $(k+1)$-king.
	
	\item There are at least one king and at least two $3$-kings in $D$.
\end{enumerate}
\end{thm}

\begin{thm} \label{main-odd}
Let $k \ge 3$ be an odd integer and $D$ a $k$-quasi-transitive digraph with a unique initial strong component $C$.   Then at least one of the following statements holds:
\begin{enumerate}
	\item Every vertex in $C$ is a $(k+1)$-king.
	
	\item If no vertex in $C$ is a $3$-king, then there are at least four $4$-kings.
\end{enumerate}
\end{thm}

The structure of the present work is the following.   In Section \ref{Tools} the necessary technical results for the rest of the paper are proved; some of this results are interesting on their own, mainly because they shed some light on the structure of $k$-quasi-transitive digraphs.   Section \ref{exist} has as cornerstones Propositions \ref{exis-deg-even-prop} and \ref{exis-deg-odd-prop}; once these propositions are proved, we focus on the distribution of $(k+1)$-kings in a $k$-quasi-transitive digraph with a unique initial component, which will help us in the next section.   In Section \ref{number}, the results of Section \ref{exist} are used to prove Theorems \ref{main-even} and \ref{main-odd}; an improvement on the minimum number of $3$-kings in quasi-transitive digraphs is given.   At last, in Section \ref{kernels}, we use our results to prove true a conjecture stated in \cite{GS-HC}: If $k \ge 3 $ is an odd integer, then every $k$-quasi-transitive digraph has a $(k+2)$-kernel.

\section{Basic Tools} \label{Tools}

The following lemmas are proved in \cite{GS-HC}.

\begin{lem} \label{q-edist}
Let $k \in \mathbb{N}$ be an even natural number, $D$ a $k$-quasi-transitive digraph and $u,v \in V(D)$ such that a $uv$-path exists.   Then:
	\begin{enumerate}
		\item If $d(u,v)=k$, then $d(v,u)=1$.
		\item If $d(u,v)=k+1$, then $d(v,u) \le k+1$.
		\item If $d(u,v) \ge k+2$, then $d(v,u)=1$
	\end{enumerate}
\end{lem}

\begin{lem} \label{q-odist}
Let $k \in \mathbb{N}$ be an odd natural number, $D$ a $k$-quasi-transitive digraph and $u,v \in V(D)$ such that a $uv$-path exists.   Then:
	\begin{enumerate}
		\item If $d(u,v)=k$, then $d(v,u)=1$.
		\item If $d(u,v)=k+1$, then $d(v,u) \le k+1$.
		\item If $d(u,v)=n \ge k+2$ with $n$ odd, then $d(v,u)=1$
		\item If $d(u,v)=n \ge k+3$ with $n$ even, then $d(v,u) \le 2$
	\end{enumerate}
\end{lem}

The next result is a simple, yet very useful, consequence of Lemmas \ref{q-edist} and \ref{q-odist}.

\begin{lem} \label{k-1-dom-comp}
Let $k \ge 2$ be an integer and $D$ a $k$-quasi-transitive digraph.   If $S_1$ and $S_2$ are distinct strong components of $D$ such that $S_1$ reaches $S_2$ in $D$, then $S_1 \stackrel{k-1}{\to} S_2$.
\end{lem}

\begin{proof}
Since $S_1$ and $S_2$ are distinct strong components of $D$ and $S_1$ reaches $S_2$, then $S_2$ cannot reach $S_1$.  Let $u \in V(S_1)$ and $v \in V(S_2)$ be arbitrarily chosen.   If $d(u,v) \ge k$, then Lemma \ref{q-edist} or \ref{q-odist}, depending on the parity of $k$, implies that $v$ reaches $u$, contradicting the previous observation.   Hence, $d(u,v) \le k-1$.
\end{proof}

We now propose some structural properties analogous to those proved by Bang-Jensen and Huang for quasi-transitive digraphs.

\begin{lem} \label{k+2-odd}
Let $k \ge 2$ be an integer and $D$ a $k$-quasi-transitive digraph.   If $P = (u_0, u_1, \dots, u_{k+1}, u_{k+2})$ is a $u_0 u_{k+2}$-path of minimum length, then $u_{k+2} \to u_{k-i}$ for every odd $i$, $1 \le i \le k$.
\end{lem}

\begin{proof}
By induction on $i$.  It follow from Lemmas \ref{q-edist} and \ref{q-odist} that $u_{k+2} \to u_0$.   Clearly, $(u_{k+2}, u_0) \cup (u_0 P u_{k-1})$ is a $u_{k+2} u_{k-1}$-path of length $k$, and thus, the $k$-quasi-transitivity of $D$ implies that $u_{k+2} \to u_{k-1}$ or $u_{k-1} \to u_{k+2}$.   But $d (u_0, u_{k+2}) = k+2$, hence $u_{k-1} \not \to u_{k+2}$ and therefore $u_{k+2} \to u_{k-1}$.

For the inductive step, let us suppose that $u_{k+2} \to u_{k-m}$ for some odd integer $1 \le m \le k-2$.   Clearly, $(u_{k+2}, u_{k-m}) \cup (u_{k-m} P u_k) \cup (u_k, u_0) \cup (u_0 P u_{k-(m+2)})$ is a $u_{k+2} u_{k-(m+2)}$-path of length $2 + k -(k-m) + k-(m+2) = k$.   The $k$-quasi-transitivity of $D$ implies that $u_{k+2} \to u_{k-(m+2)}$ or $u_{k-(m+2)} \to u_{k+2}$.   Again, $d(u_0, u_{k+2}) = k+2$ and thus, $u_{k-(m+2)} \not \to u_{k+2}$.   Hence, $u_{k+2} \to u_{k-(m+2)}$.

The result now follows from the Principle of Mathematical Induction.
\end{proof}

\begin{lem} \label{k+1-even}
Let $k \ge 2$ be an integer and $D$ a $k$-quasi-transitive digraph.   If $P = (u_0, u_1, \dots, u_{k+1}, u_{k+2})$ is a $u_0 u_{k+2}$-path of minimum length, then $u_{k+1} \to u_{k-i}$ for every even $i$, $2 \le i \le k$.
\end{lem}

\begin{proof}
By induction on $i$.   It follow from Lemmas \ref{q-edist} and \ref{q-odist} that $u_{k+2} \to u_0$.   Clearly, $(u_{k+1}, u_{k+2}, u_0) \cup (u_0 P u_{k-2})$ is a $u_{k+1} u_{k-2}$-path of length $k$, and thus, the $k$-quasi-transitivity of $D$ implies that $u_{k+1} \to u_{k-2}$ or $u_{k-2} \to u_{k+1}$.   But $d (u_0, u_{k+2}) = k+2$, hence $u_{k-2} \not \to u_{k+1}$ and therefore $u_{k+1} \to u_{k-2}$.

For the inductive step, let us suppose that $u_{k+1} \to u_{k-m}$ for some even integer $2 \le m \le k-2$.  Clearly, $(u_{k+1}, u_{k-m}) \cup (u_{k-m} P u_k) \cup (u_k, u_0) \cup (u_0 P u_{k-(m+2)})$ is a $u_{k+1} u_{k-(m+2)}$-path of length $2 + k-(k-m) + k-(m+2) = k$.   The $k$-quasi-transitivity of $D$ implies that $u_{k+1} \to u_{k-(m+2)}$ or $u_{k-(m+2)} \to u_{k+1}$.   Again, $d(u_0, u_{k+2}) = k+2$ and thus, $u_{k-(m+2)} \not \to u_{k+1}$.   Hence, $u_{k+1} \to u_{k-(m+2)}$.

The result now follows from the Principle of Mathematical Induction.
\end{proof}

\begin{lem} \label{k+2-even}
Let $k \ge 2$ be an even integer and $D$ a $k$-quasi-transitive digraph.   If $P = (u_0, u_1, \dots, u_{k+1}, u_{k+2})$ is a $u_0 u_{k+2}$-path of minimum length, then $u_{k+2} \to u_{k-i}$ for every $0 \le i \le k$.
\end{lem}

\begin{proof}
It has been proved in Lemma \ref{k+2-odd} that $u_{k+2} \to u_{k-i}$ for every odd $i$, $1 \le i \le k-1$.   For the remaining cases, we will proceed by induction on $i$ for even $0 \le i \le k$.    It follows from Lemma \ref{k+2-odd} that $u_{k+2} \to u_{k-(k-1)} = u_1$ and thus $(u_{k+2}, u_1) \cup (u_1 P u_k)$ is a $u_{k+2} u_k$-path of length $k$ and hence $u_{k+2} \to u_k$ or $u_k \to u_{k+2}$.   But $d(u_0, u_{k+2}) = k+2$, thus, $u_k \not \to u_{k+2}$ and thence $u_{k+2} \to u_k$.   Also, $(u_{k+2}, u_k, u_0) \cup (u_0 P u_{k-2})$ is a $u_{k+2} u_{k-2}$-path of length $k$.   Again, $u_{k+2} \to u_{k-2}$ or $u_{k-2} \to u_{k+2}$, but the choice of $P$ of minimum length implies that $u_{k+2} \to u_{k-2}$.

If $u_{k+2} \to u_{k-m}$ for even $m$, $4 \le m \le k-2$, it is clear that $(u_{k+2}, u_{k-m}) \cup (u_{k-m} P u_k) \cup (u_k, u_0) \cup (u_0 P u_{k-(m+2)})$ is a $u_{k+2} u_{k-(m+2)}$-path of length $2 + k-(k-m) + k-(m+2) = k$.   The $k$-quasi-transitivity of $D$ and the choice of $P$ imply that $u_{k+2} \to u_{k-(m+2)}$.

The result now follows from the Principle of Mathematical Induction. 
\end{proof}

The last results of this section are our main tools to prove the existence of a $(k+1)$-king in a $k$-quasi-transitive strong digraph, and hence, in an arbitrary $k$-quasi-transitive digraph with a unique initial strong component.

\begin{lem} \label{deg-even}
Let $k \ge 2$ be an even integer and $D$ a $k$-quasi-transitive digraph.   If $P = (u = u_0, u_1, \dots, u_{k+1}, u_{k+2} = v)$ is a $u v$-path of minimum length, then $d^+(v) \ge d^+(u) + k$.
\end{lem}

\begin{proof}
Lemma \ref{q-edist} implies that $v \to u$.   First, suppose that $k = 2$.   If $w \in N^+ (u)$, then $(v,u,w)$ is a path of length $2$ in $D$, but $D$ is $2$-quasi-transitive, thus, $v \to w$ or $w \to v$.   But $d(u,v) = k+2 = 4$, hence $w \not \to v$, otherwise $(u, w, v)$ would be a $uv$-path in $D$ shorter than $P$, resulting in a contradiction.   Therefore, $v \to w$ and $N^+(u) \subseteq N^+(v)$.   Lemma \ref{k+2-even} implies that $v \to u_2$ and we have already observed that $v \to u$.   But the choice of $P$ of minimum length implies that $u \not \to u_2$ and $D$ is loopless, so $u \not \to u$.   We conclude that $d^+(v) \ge d^+(u) + 2 = d^+(u) + k$.

If $k \ge 4$, then Lemmas \ref{k+2-even} and \ref{q-edist} imply $v \to u_3$ and $u_k \to u$, respectively.   Thus, if $w \in N^+(u)$, we can consider $P' = (v, u_3) \cup (u_3 P u_k) \cup (u_k, u, w)$.   Since $w \ne u_i$ for $2 \le i \le k$, $P'$ is a path of length $k$ in $D$, thus $v \to w$ or $w \to v$.   Again, $d(u,v) = k+2$, from where we can derive $w \not \to v$.   Hence $N^+ (u) \subseteq N^+ (v)$.   Once again, it follows from Lemma \ref{k+2-even} that $v \to u_i$ for every $0 \le i \le k$.   Also, the choice of $P$ as a $uv$-path of minimum length and the fact that $D$ is loopless imply that $u \not \to u_i$ with $i \in \{ 0, 2, 3, \dots, k \}$.   Thence, $d^+(v) \ge d^+(u) + k$.
\end{proof}

\begin{lem} \label{deg-odd}
Let $k \ge 3$ be an odd integer and $D$ a $k$-quasi-transitive digraph.   If $P = (u_0, u_1, \dots, u_{k+1}, u_{k+2})$ is a $u_0 u_{k+2}$-path of minimum length, then $d^+(u_{k+1}) \ge d^+(u_0) + \frac{k-1}{2}$.
\end{lem}

\begin{proof}
Lemma \ref{q-edist} implies that $\{ u_{k+2}, u_k \} \to u_0$.   If $w \in N^+(u_0)$, then $w \ne u_i$ for every $2 \le i \le k+2$.

If $k = 3$, then $(u_{k+1}, u_{k+2}, u_0, w)$ is a $u_{k+1} w$-path of length $3$.   The $k$-quasi-transitivity of $D$ and the minimality of length of $P$ imply that $u_{k+1} \to w$ and hence, $N^+ (u_0) \subseteq N^+ (u_{k+1})$.   Also, $u_{k+1} \to u_{k+2}$ but $u_0 \not \to u_{k+2}$, therefore $d^(u_{k+1}) \ge d^+(u_0) + 1 = d^+(u_0) + \frac{k-1}{2}$.

Now, we can assume that $k \ge 5$.   Lemma \ref{k+1-even} gives us $u_{k+1} \to u_{k-i}$ for every even $i$ such that $2 \le i \le k$, and hence, $u_{k+1} \to u_3$.   So, we can consider the $u_{k+1} w$-path $(u_{k+1}, u_3) \cup (u_3 P u_k) \cup (u_k, u_0, w)$ of length $k$.   The $k$-quasi-transitivity of $D$ together with the choice of $P$ imply that $u_{k+1} \to w$, therefore $N^+(u_0) \subseteq N^+(u_{k+1})$.

Since $k$ is an odd integer, then $u_{k+1} \to u_i$ for every odd $i$ such that $3 \le i \le k-2$.   Also, $u_{k+1} \to u_{k+2}$.   Recalling that $P$ is a $u_0 u_{k+2}$-path of minimum length, we can conclude that $u_0 \not \to u_j$ for every $j \ge 2$.   Hence $d^+(u_{k+1}) \ge d^+(u_0) + \frac{k-1}{2}$.
\end{proof}

\begin{lem} \label{dom-k+2-even}
Let $k \ge 2$ be an even integer and $D$ a $k$-quasi-transitive digraph.   If $P = (u_0, u_1, \dots, u_{k+2})$ is a $u_0 u_{k+2}$-path of minimum length in $D$, then $u_{k+2} \to w$ for every $w \in V(D)$ such that $d(u_0, w) \le k$.
\end{lem}

\begin{proof}
By Lemma \ref{k+2-even}, $u_{k+2} \to u_i$ for every $0 \le i \le k$.   In particular, $u_{k+2} \to u_0$, so we can consider $w \in V(D)$ such that $0 < d(u_0, w) = n \le k$ and $P' = (u_0, w_1, w_2, \dots, w_n = w)$ a $u_0 w$-path realizing the distance between $u_0$ and $w$.

If $k \ne 2$ and $n \le k-2$, then $W = (u_{k+2}, u_{n+2}) \cup (u_{n+2} P u_k) \cup (u_k, u_0) \cup (u_0 P' w)$ is a $u_{k+2} w$-walk of length $k$.   But $d(u_0, w) = n$ and $d(u_0, x) > n$ for each $x \in V(u_{n+2} P u_k)$.   Hence, $V(u_0 P' w) \cap V(u_{n+2} P u_k) = \varnothing$ and $W$ is a $u_{k+2} w$-path of length $k$.   If $n = k-1$, then $(u_{k+2}, u_0) \cup (u_0 P' w)$ is a $u_{k+2} w$-path of length $k \in D$.   If $n = k$, then, by the previous cases, $u_{k+2} \to w_1$ and thus $(u_{k+2}, w_1) \cup (w_1 P' w)$ is a $u_{k+2} w$-path of length $k$.

We have already shown that if $d(u_0, w) \le k$, then a $u_{k+2} w$-path of length $k$ exists in $D$.   Since $D$ is $k$-quasi-transitive, we have that $u_{k+2} \to w$ or $w \to u_{k+2}$.   But $d(u_0, u_{k+2}) = k+2$, hence $w \not \to u_{k+2}$ and therefore $u_{k+2} \to w$.
\end{proof}

\begin{lem} \label{dom-k+2-odd}
Let $k \ge 3$ be an odd integer and $D$ a $k$-quasi-transitive digraph.   If $P = (u_0, u_1, \dots, u_{k+2})$ is a $u_0 u_{k+2}$-path of minimum length in $D$, then $u_{k+2} \to w$ for every $w \in V(D)$ such that $d(u_0, w) = n \le k-1$, with $n$ even.
\end{lem}

\begin{proof}
It follows from Lemma \ref{k+2-odd} that $u_{k+2} \to u_i$ for every even $i$ such that $0 \le i \le k$.   In particular, $u_{k+2} \to u_0$, so we can consider $w \in V(D)$ such that $0 < d(u_0, w) = n \le k$, $n$ even and $P' = (u_0, w_1, w_2, \dots, w_n = w)$ a $u_0 w$-path realizing the distance between $u_0$ and $w$.

If $k \ne 3$ and $n \le k-3$, then $W = (u_{k+2}, u_{n+2}) \cup (u_{n+2} P u_k) \cup (u_k, u_0) \cup (u_0 P' w)$ is a $u_{k+2} w$-walk of length $k$.   But $d(u_0, w) = n$ and $d(u_0, x) > n$ for each $x \in V(u_{n+2} P u_k)$.   Hence, $V(u_0 P' w) \cap V(u_{n+2} P u_k) = \varnothing$ and $W$ is a $u_{k+2} w$-path of length $k$.   If $n = k-1$, then $(u_{k+2}, u_0) \cup (u_0 P' w)$ is a $u_{k+2} w$-path of length $k$ in $D$.

So, if $d(u_0, w)$ is even and less than or equal to $k$, then a $u_{k+2} w$-path of length $k$ exists in $D$.   Since $D$ is $k$-quasi-transitive, we have that $u_{k+2} \to w$ or $w \to u_{k+2}$.   But $d(u_0, u_{k+2}) = k+2$, hence $w \not \to u_{k+2}$ and therefore $u_{k+2} \to w$.
\end{proof}

\section{Existence Results} \label{exist}

We begin this section proving that, for every integer $k \ge 2$, every $k$-quasi-transitive strong digraph has a $(k+1)$-king.   It is important to remark that our result is not merely existential, we describe which vertices of the digraph are $(k+1)$-kings.

\begin{lem} \label{exis-deg-even}
Let $k \ge 2$ be an even integer and $D$ a $k$-quasi-transitive strong digraph.   If $v \in V(D)$ is such that $\Delta^+_D \ge d^+(v) > \Delta^+_D - k$, then $v$ is a $(k+1)$-king. 
\end{lem}

\begin{proof}
Let $v \in V(D)$ be arbitrarily chosen such that $\Delta^+_D \ge d^+(v) \ge \Delta^+_D - k$.    Since $D$ is strong, if $v$ is not a $(k+1)$-king, there must exist $u \in V(D)$ such that $d(v,u) = k+2$.   But Lemma \ref{deg-even} gives us $d^+(u) \ge d^+(v) + k > (\Delta^+_D - k) + k = \Delta^+_D$, which results in a contradiction.   Since the contradiction arose from assuming that $v$ is not a $(k+1)$-king, it must be the case that $v$ is a $(k+1)$-king.
\end{proof}

\begin{lem} \label{exis-deg-odd}
Let $k \ge 3$ be an odd integer and $D$ a $k$-quasi-transitive strong digraph.   If $v \in V(D)$ is such that $\Delta^+_D \ge d^+(v) > \Delta^+_D - \frac{k-1}{2}$, then $v$ is a $(k+1)$-king. 
\end{lem}

\begin{proof}
Let $v \in V(D)$ be arbitrarily chosen such that $\Delta^+_D \ge d^+(v) \ge \Delta^+_D - \frac{k-1}{2}$.    Since $D$ is strong, if $v$ is not a $(k+1)$-king, there must exist $u \in V(D)$ such that $d(v,u) = k+2$.   Let $(v= u_0, u_1, \dots, u_{k+1}, u_{k+2} = u)$ be a path in $D$.   But Lemma \ref{deg-odd} gives us $d^+(u_{k+1}) \ge d^+(v) + \frac{k-1}{2} > (\Delta^+_D - \frac{k-1}{2}) + \frac{k-1}{2} = \Delta^+_D$, which results in a contradiction.   Since the contradiction arose from assuming that $v$ is not a $(k+1)$-king, it must be the case that $v$ is a $(k+1)$-king.
\end{proof}

We are now ready to prove Propositions \ref{exis-deg-even-prop} and \ref{exis-deg-odd-prop}.   Again, we do not simply prove the existence of $(k+1)$-kings in a $k$-quasi-transitive digraph $D$, we observe that through a simple exploration of the out-degrees of the initial component of $D$, we can easily find such $(k+1)$-kings.   As a matter of fact, since constructing the condensation of $D$ can be done in linear time, a $(k+1)$-king can be found in linear time in a $k$-quasi-transitive digraph.

\begin{proof}[Proof of Proposition \ref{exis-deg-even-prop}]
Let $v \in V(C)$ be such that $\Delta_C^+ \ge d^+ (v) > \Delta_C^+ (v) - k$.   By Lemma \ref{exis-deg-even}, $v$ is a $(k+1)$-king of $C$.   Recalling that every strong component of $D$ must be reached by some initial strong component and $C$ is the unique initial strong component, Lemma \ref{k-1-dom-comp} implies that $v$ is a $(k+1)$-king of $D$.
\end{proof}

\begin{proof}[Proof of Proposition \ref{exis-deg-odd-prop}]
Let $v \in V(C)$ be such that $\Delta_C^+ \ge d^+ (v) > \Delta_C^+ (v) - \frac{k-1}{2}$.   By Lemma \ref{exis-deg-odd}, $v$ is a $(k+1)$-king of $C$.   Recalling that every strong component of $D$ must be reached by some initial strong component and $C$ is the unique initial strong component, Lemma \ref{k-1-dom-comp} implies that $v$ is a $(k+1)$-king of $D$.
\end{proof}

Now, we get as an immediate consequence Theorem \ref{genBJ}.

\begin{proof}[Proof of Theorem \ref{genBJ}]
The non-trivial implication is a direct consequence of Propositions \ref{exis-deg-even-prop} and \ref{exis-deg-odd-prop}.
\end{proof}

Now we know that every $k$-quasi-transitive digraph $D$ has a $(k+1)$-king, we can analyze how this kings are distributed through the unique initial component of $D$.   This analysis will also help us to find sufficient conditions for the existence of $2$, $3$ and $4$-kings in a $k$-quasi-transitive digraph.

\begin{thm} \label{k+2-3-king}
Let $k \ge 2$ be an even integer and $D$ a $k$-quasi-transitive digraph.   If $v \in V(D)$ is a $(k+2)$-king and $u \in V(D)$ is such that $d(v,u) = k+2$, then $u$ is a $3$-king of $D$.  Moreover, if $w \in A(D)$ is such that $d(u,w) = 3$, then $d(v,w) = k+2$ and $w$ is also a $3$-king. 
\end{thm}

\begin{proof}
Let $v \in V(D)$ be a $(k+2)$-king of $D$ and  $u \in V(D)$ such that $d(v,u) = k+2$.    It follows from Lemma \ref{dom-k+2-even} that $u \to w$ for every $w \in V(D)$ such that $d(v,w) \le k$.   Let $w_{k+i}$ be a vertex such that $w_{k+i} \ne u$ and $d(v, w_{k+i}) = k+i$, $0 < i$.  Since $v$ is a $(k+2)$-king, $i \in \{ 1, 2 \}$.   Let $P = (v, w_1, \dots, w_{k+i})$ a $v w_{k+i}$-path realizing the distance from $v$ to $w_{k+i}$.   By the choice of $P$, $d(v, w_k) = k$, so, Lemma \ref{dom-k+2-even} implies that $u \to w_k$, hence, $(u, w_k) \cup (w_k P w_{k+i})$ is a path of length $i + 1 \le 3$.   Thus, $d(u,w) \le 3$ for every $w \in V(D)$, concluding that $u$ is a $3$-king.   Since $d(u, w_{k+i}) \le i+1$, if $d(u, w) = 3$, it must happen that$d(v,w) = k+2$, otherwise,  $d(u,w) < 3$.
\end{proof}

\begin{cor} \label{2k-even}
Let $k \ge 4$ be an even integer and $D$ a $k$-quasi-transitive digraph.   If $v \in V(D)$ is a $(k+2)$-king but it is not a $(k+1)$-king, then there is a vertex $u \in V(D)$ such that $d(v,u) = k+2$ and is a $2$-king.   Also, the set $S = \left\{ w \in V(D) \colon\ d(v,w) = k+2 \right\}$ induces a semicomplete subdigraph of $D$.
\end{cor}

\begin{proof}
In virtue of Theorem \ref{k+2-3-king}, it suffices to prove that the $S$ induces a semicomplete subdigraph of $D$, thus, a $2$-king $u$ exists in $D[S]$ that is also a $2$-king of $D$, because $u \in S$ implies that $u$ $2$-dominates every vertex in $V(D) \setminus S$.

Since $v$ is a $(k+2)$-king but it is not a $(k+1)$-king, there exists a vertex $w$ such that $d(v,w) = k+2$.   It follows from Theorem \ref{k+2-3-king} that $w$ is a $3$-king and that the only vertices $x \in V(D) \setminus \{ w \}$ such that $d(w,x) = 3$ are precisely those vertices in $S$.   If $|S| = 1$, then $w$ is a $2$-king of $D$.   Otherwise, let $x \ne y \in S$ be arbitrarily chosen.   Hence, there is a path $(v, y_1, y_2, \dots, y_{k+2} = y)$ realizing the distance from $v$ to $y$.   Lemma \ref{dom-k+2-even} implies that $x \to u$ for every $u \in V(D)$ such that $d(v,u) \le k$, in particular, and given the fact that $k \ge 4$, $x \to y_3$ and hence $(x, y_3, y_4, \dots, y_{k+1}, y)$ is an $xy$-path of length $k$.   The $k$-quasi-transitivity of $D$ implies that an arc exists between $x$ and $y$.   Since $x$ and $y$ were arbitrarily chosen, $D[S]$ is a semicomplete digraph.   Theorem \ref{sc-2k} implies that a $2$-king $u$ exists in $D[S]$.
\end{proof}

\begin{cor} \label{even-k+1-2-fin}
Let $k \ge 4$ be an even integer and $D$ a $k$-quasi-transitive digraph with a unique initial strong component $C$.   One of the following statements hold:

\begin{enumerate}
	\item Every vertex of $C$ is a $(k+1)$-king.
	
	\item There are vertices $u_1, u_2, u_3 \in V(C)$ such that $u_1$ is a $(k+1)$-king of $D$, $u_2$ is a $(k+2)$-king of $D$, $u_2 \to u_1$, $d(u_2, 					u_3) = k+2$ and $u_3$ is a $2$-king in $D$.   Also, every vertex at distance $k+2$ from $u_2$ is a $3$-king.
\end{enumerate}
\end{cor}

\begin{proof}
Lemma \ref{exis-deg-even-prop} implies that a $(k+1)$-king exists in $D$.   If $u' \in V(C)$ is not a $(k+1)$-king of $D$, then we can choose a $(k+1)$-king of $D$ $u \in V(C)$ such that $d(u',u) \le d(u',x)$ for every $x$ $(k+1)$-king of $D$.   Let $P = (u', \dots, v,u)$ be a $u'u$-path of minimum length.   We can deduce from the choice of $u$ that $v$ is not a $(k+1)$-king, because $d(u', v) < d(u', u)$.   But $v \to u$ and $u$ is a $(k+1)$-king, implying that $v$ is a $(k+2)$-king.   It follows from Corollary \ref{2k-even} that a vertex $w$ exists in $V(C)$ such that $d(v,w) = k+2$ and $w$ is a $2$-king.   Our result is now proved for $u_1 = u$, $u_2 = v$ and $u_3 = w$.   The last statement follows directly from Theorem \ref{k+2-3-king}.
\end{proof}

Let us point out that Corollary \ref{even-k+1-2-fin} gives us two possibilities for a $k$-quasi-transitive digraph with even $k \ge 4$ and a unique initial strong component $C$:   every vertex in $C$ ia $(k+1)$-king of $D$ or there exists a $2$-king in $D$.   The result does not hold for $k = 2$, but in \cite{BJ-H2} the existence of $2$-kings in quasi-transitive digraphs was characterized.   As we mentioned in the introduction, we can summarize the results for the existence of $(k+1), 3$ and $2$-kings in $k$-quasi-transitive digraphs, for even $k \ge 4$, in Theorem \ref{main-even}.

\begin{proof}[Proof of Theorem \ref{main-even}]
Let us assume that there is at least one vertex that is not a $(k+1)$-king.   It follows from Corollary \ref{even-k+1-2-fin} that a $2$-king $u$ exists in $C$.   Since not every vertex of $C$ is a $(k+1)$-king, and $C$ is strong, there must be at least one vertex $v \in V(C)$ such that $v \to u$.   Clearly $v$ is a $3$-king.
\end{proof}

\begin{thm} \label{k+2-4-king}
Let $k \ge 3$ be an odd integer and $D$ a $k$-quasi-transitive digraph.   If $v \in V(D)$ is a $(k+2)$-king and $u \in V(D)$ is such that $d(v,u) = k+2$, then $u$ is a $4$-king of $D$.  Moreover, if $w \in A(D)$ is such that $d(u,w) = 4$, then $d(v,w) = k+2$ and $w$ is also a $4$-king.
\end{thm}

\begin{proof}
Let $v \in V(D)$ be a $(k+2)$-king and $u \in V(D)$ such that $d(v,u) = k+2$.   It follows from Lemma \ref{dom-k+2-odd} that $u \to w$ for every $w \in V(D)$ such that $d(v,w) \le k$ is even.   Hence, if $d(v,w) \le k$ is odd, then $d(u,w) \le 2$.   Let $w_{k+i}$ be a vertex such that $w_{k+i} \ne u$ and $d(v, w_{k+i}) = k+i$, $0 < i$.  Since $v$ is a $(k+2)$-king, $i \in \{ 1, 2 \}$.   Let $P = (v, w_1, \dots, w_{k+i})$ a $v w_{k+i}$-path realizing the distance from $v$ to $w_{k+i}$.   By the choice of $P$, $d(v, w_k) = k$, so, Lemma \ref{dom-k+2-odd} implies that $u \to w_{k-1}$, hence, $(u, w_{k-1}) \cup (w_{k-1} P w_{k+i})$ is a path of length $i + 2 \le 4$.   Thus, $d(u,w) \le 4$ for every $w \in V(D)$, concluding that $u$ is a $4$-king.   We can also conclude that if $d(v,w) < k+2$, then $d(u,w) < 4$, proving the second statement of the theorem.
\end{proof}

\begin{cor} \label{odd-k+1-4-fin}
Let $k \ge 3$ be an odd integer and $D$ a $k$-quasi-transitive digraph with a unique initial strong component $C$.   One of the following statements hold:

\begin{enumerate}
	\item Every vertex of $C$ is a $(k+1)$-king.
	
	\item There are vertices $u_1, u_2, u_3 \in V(C)$ such that $u_1$ is a $(k+1)$-king of $D$, $u_2$ is a $(k+2)$-king of $D$, $u_2 \to u_1$, $d(u_2, u_3) = k+2$ and
	 $u_3$ is a $4$-king in $D$.   Also, every vertex at distance $k+2$ from $u_2$ is a $4$-king.
\end{enumerate}
\end{cor}

\begin{proof}
Lemma \ref{exis-deg-odd-prop} implies that a $(k+1)$-king exists in $D$.   If $u' \in V(C)$ is not a $(k+1)$-king of $D$, then we can choose a $(k+1)$-king of $D$ $u \in V(C)$ such that $d(u',u) \le d(u',x)$ for every $x$ $(k+1)$-king of $D$.   Let $P = (u', \dots, v,u)$ be a $u'u$-path of minimum length.   We can deduce from the choice of $u$ that $v$ is not a $(k+1)$-king, because $d(u', v) < d(u', u)$.   But $v \to u$ and $u$ is a $(k+1)$-king, implying that $v$ is a $(k+2)$-king.   By the choice of $v$ we affirm that there exists a vertex $w \in V(C)$ such that $d(v,w) = k+2$.   It follows from Theorem \ref{k+2-4-king} that $w$ is a $4$-king.   Our result is now proved for $u_1 = u$, $u_2 = v$ and $u_3 = w$.   The last statement follows directly from Theorem \ref{k+2-4-king}.
\end{proof}

From here, it is easy now to prove Theorem \ref{main-odd}.

\begin{proof}[Proof of Theorem \ref{main-odd}]
	Let us assume that not every vertex of $C$ is a $(k+1)$-king of $D$ and there are no $3$-kings in $D$.   Let $u_1, u_2$ and $u_3$ be like in Corollary \ref{odd-k+1-4-fin}.   Since there are no $3$-kings in $D$, there must be at least one vertex $v_1$ such that $d(u_3, v_1) = 4$.   By Theorem \ref{k+2-4-king}, $v_1$ is a $4$-king of $D$.   It follows from Theorem \ref{k+2-4-king} that $d(u_2, v_1) = k+2$.   Let us suppose that $v_1$ and $v_2 = u_3$ are the only vertices of $D$ at distance $k+2$ from $u_2$.   Let $P_1 = (u_2 = x_0, \dots, x_{k+1}, x_{k+2}=v_1)$ and $P_2 = (u_2 = y_0, \dots, y_{k+1}, y_{k+2}=v_2)$ be paths in $D$.   Since $v_1$ and $v_2$ are $4$-kings of $D$, then $x_{k+1}$ and $y_{k+1}$ are $5$-kings of $D$.   Moreover, by Theorem \ref{k+2-4-king} $v_2$ is the only vertex at distance $4$ from $v_1$ and $v_1$ is the only vertex at distance $4$ from $v_2$.   Hence, every vertex is at distance $4$ from $x_{k+1}$ (resp $y_{k+1}$), except perhaps $v_2$ (resp. $v_1$).
	
	If  $x_{k+1} \ne y_{k+1}$ and $k > 3$, then by Lemma \ref{dom-k+2-odd} we have $v_1 \to y_4$, hence $(x_{k+1}, v_1, y_4) \cup (y_4 P_1 v_2)$ is an $x_{k+1} v_2$-path of length $k$ in $D$.   The $k$-quasi-transitivity of $D$ and the fact that $d(v_2, v_1) = 4$ implies that $x_{k+1} \to v_2$.   A similar argument shows that $y_{k+1} \to v_1$.   Hence, $x_{k+1}$ and $y_{k+1}$ are $4$-kings of $D$.
	
	If  $x_{k+1} \ne y_{k+1}$ and $k = 3$, using again Lemma \ref{dom-k+2-odd} we have $v_1 \to y_2$, hence $(v_1, y_2, y_3, y_4)$ is a $v_1 y_4$-path of length $3$.   Since $d(v_1, v_2) = 5$ and $D$ is $3$-quasi-transitive, we have $y_4 \to v_1$.   A similar argument shows that $x_4 \to v_2$.   Again, $x_4$ and $y_4$ are $4$-kings of $D$.
	
	If $x_{k+1} = y_{k+1}$ then there are no vertices at distance $5$ from $x_{k+1}$, thus it is a $4$-king.   Let us observe that $d(v_1, x_{k+1}) = 3$.    If $d(v_1, x_{k+1}) = 4$, Theorem \ref{k+2-4-king} would imply that $d(u_2, x_{k+1}) = k+2$, a contradiction.   Also $d(v_1, x_{k+1}) > 2$ because $d(v_1, v_2) = 4$.   If $x_{k+1}$ is the only vertex in $D$ such that $d(v_1, x_{k+1}) = 3$, then $x_{k+1}$, since $x_{k+1} \to v_2$, we would have that $x_{k+1}$ is a $3$-king, a contradiction.   Let $v_3$ be a vertex in $D$ such that $v_3 \ne x_{k+1}$ and $d(v_1, v_3) = 3$.   If $d(u_2, v_3) \le k$, then Lemma \ref{dom-k+2-odd} would imply that $d(v_1, v_3) \le 2$.   Recalling that $u_2$ is a $(k+2)$-king and using Corollary \ref{odd-k+1-4-fin} we can conclude that $d(u_2, v_3) = k+1$.  So, let $P_3 = (u_2 = z_0, \dots, z_{k+1} = v_3)$ be a path.   Using Lemma \ref{dom-k+2-odd} again, we obtain $v_1 \to z_2$, thence, $(v_1, z_2) \cup (z_2 P_3 v_3)$ is a $v_1 v_3$-path of length $k$.   Since $d(v_1, v_3) = 3$, we can conclude from the $k$-quasi-transitivity of $D$ that $v_3 \to v_1$.   A similar argument shows that a $v_2 v_3$-path of length $k$ exists in $D$, from here, using that $d(v_2, v_1) = 4$ and the $k$-quasi-transitivity of $D$, we can conclude that $v_3 \to v_4$.   Therefore, $v_3$ is a $4$-king of $D$.
	
	Now, let us suppose that there are exactly three vertices $v_1$, $v_2 = u_3$ and $v_3$ at distance $k+2$ from $u_2$.   We can consider the paths $P_1 = (u_2 = x_0, \dots, x_{k+1} = w_1, x_{k+2}=v_1)$, $P_2 = (u_2 = y_0, \dots, y_{k+1} = w_2, y_{k+2}=v_2)$ and $P_3 = (u_2 = z_0, \dots, z_{k+1} = w_3, z_{k+2}=v_3)$ in $D$.   Again, $w_1, w_2, w_3$ are $5$-kings of $D$.   Since there are not $3$-kings in $D$, for each $v_i$, $1 \le i \le 3$, there must be at least one vertex at distance $4$.   We can assume that this vertex is one of the remaining $v_j$, $j \ne i$, otherwise, Theorem \ref{k+2-4-king} would give us the existence of a fourth $4$-king.   If $w_1 = w_2 = w_3$, then $w_1$ is a $4$-king, and we are done.
	
	By Lemma \ref{dom-k+2-odd} we have $v_1 \to y_4$, hence $(w_1, v_1, y_4) \cup (y_4 P_1 v_2)$ is a $w_1 v_2$-path of length $k$ in $D$; the $k$-quasi-transitivity of $D$ implies that $w_1 \to v_2$ or $v_2 \to w_1$.   By analogous arguments we can prove that there are arcs between $w_i$ and $v_j$ for $i \ne j$. 
	
	First, we will assume that $w_1, w_2, w_3$ are pairwise distinct.   Let us suppose without loss of generality that $d(v_2, v_1) = 4$, this implies that $w_1 \to v_2$.   If $w_1 \to v_3$, then $w_1$ is a $4$-king.   Otherwise, we can consider two cases.   If $v_2 \to w_3$, then $(w_1, v_2, w_3, v_3)$ is a path in $D$, hence $d(w_1, v_3) \le 4$, and again, $w_1$ is a $4$-king.   If $w_3 \to v_2$, then $(w_3, v_3, w_1, v_1)$ is a path in $D$, implying that $d(w_3, v_1) \le 4$; hence $w_3$ is a $4$-king of $D$.
	
	 IF $w_2 = w_3$, then it cannot be the case that $v_1 \to w_2$, if so, $d(v_1, v_i) \le 2$ for $i \in \{ 1, 2 \}$, contradicting that $d(v_1, v_i) = 4$ for some $i \in \{ 1, 2 \}$.   Hence $w_2 \to v_1$ and it turns out that $w_2$ is a $4$-king of $D$.   Cases $w_1 = w_2$ and $w_3 = w_1$ can be dealt similarly.

\end{proof}

\begin{prop}
Let $k \ge 2$ be an integer and $D$ a $k$-quasi-transitive digraph.   If $u$ is a $(k+1)$-king of $D$, then every vertex at distance $k+1$ from $u$ is also a $(k+1)$-king of $D$.
\end{prop}

\begin{proof}
Let $v \in V(D)$ be such that $d(u,v) = k+1$ and $P = (u = u_0, u_1, \dots, u_{k+1} = v)$ be a $uv$-path.   We will prove by induction on $n$ that $d(v,w) = n \le k+1$ implies $d(v,w) \le k+1$.   The case $n = 0$ follows from Lemmas \ref{q-edist} and \ref{q-odist}.   Let $w \in V(D)$ be such that $d(u, w) = n + 1 \le k+1$ and let $P' = (u = w_0, w_1, \dots, w_{n+1} = w)$ be a $uw$-path in $D$.

The Induction Hypothesis gives us that $d(v, w_n) \le k+1$.   If $d(v,w_n) < k+1$, then $d(v, w) \le k+1$, because $w_n \to w$.   So, we can assume that $d(v,w_n) = k+1$.

Let us suppose that $d(v, w) > k+1$.   Since $d(v,w_n) = k+1$ and $w_n \to w$, then $d(v,w)=k+2$.   It can be deduced from Lemmas \ref{q-edist} and \ref{q-odist} that $d(w,v) = 1$.   If $n < k-1$, then $k+1 = d(u,v) \le d(u,w) + d(w,v) = n+2 < k+1$, resulting in a contradiction.   Hence, $n \ge k-1$.   If $n = k-1$, then $w = w_{n+1} = w_k$ and $(w_1 P' w_k) \cup (w_k, v)$ is a $w_1 v$-path of length $k$.   The $k$-quasi-transitivity of $D$ implies that $w_1 \to v$ or $v \to w_1$.   If $w_1 \to v$, then $(u, w_1, v)$ is a path of length $2$, contradicting that $d(u,v) = k+1 \ge 3$.   If $v \to w_1$, then $(v,w_1) \cup (w_1 P' w_k)$ is a $vw$-path of length $k$, hence $d(v,w) \le k+1$.   If $n = k \ge 3$, then $w = w_{n+1} = w_{k+1}$ and $(w_2 P' w_{k+1}) \cup (w_{k+1}, v)$ is a $w_2 v$-path of length $k$.   The $k$-quasi-transitivity of $D$ implies that $w_2 \to v$ or $v \to w_2$.   If $w_2 \to v$, then $(u, w_1, w_2, v)$ is a path of length $3$, contradicting that $d(u,v) = k+1 \ge 4$.   If $v \to w_2$, then $(v,w_2) \cup (w_2 P' w_{k+1})$ is a $vw$-path of length $k$, hence $d(v,w) \le k+1$.   If $n = k = 2$, then $D$ is quasi-transitive and the path $(w_2, w, v)$ implies that $v \to w_2$ or $w_2 \to v$.   If $v \to w_2$, then $(v, w_2, w)$ is a $vw$-path of length $2 < k+1$.   If $w_2 \to v$, then $(w_1, w_2, v)$ is a $w_1 v$-path of length $2$, the quasi-transitivity of $D$ implies that $w_1 \to v$ or $v \to w_1$.   But the former case cannot occur, otherwise $(u, w_1, v)$ would be a $uv$-path of length $2 < k+1$.   Hence $v \to w_1$ and $(v, w_1, w_2, w)$ is a $vw$-path of length $3 = k+1$.   In every case, a contradiction arises from the assumption that $d(v,w) > k+1$, thence, $d(v,w) \le k+1$.
\end{proof}

\section{On the minimum number of $(k+1)$-kings} \label{number}

The main observations on the number of $(k+1)$-kings in a $k$-quasi-transitive digraph with a unique initial strong component are collected in the results of this section.   We divide them in two cases, even and odd values of $k$; as it is usual with $k$-quasi-transitive digraphs, the results for the even case are slightly better than the odd case.

\begin{thm} \label{k+2-k+1}
Let $k \ge 4$ be an even integer and $D$ a $k$-quasi-transitive digraph with a unique initial strong component $C$ such that $|V(C)|=n$, then:
\begin{enumerate}
	\item If $n \le k$, then the exact number of $(k-1)$-kings of $D$ is $n$.
	
	\item If $n = k+1$, then the exact number of $k$-kings of $D$ is $k+1$.
	
	\item If $n \ge k+2$, then the number of $(k+1)$-kings of $D$ is at least $k+2$.
\end{enumerate}

\noindent
Also, in the last case, the $(k+1)$-kings are distributed on a single path of length $k+1$ or, if $n \ge 3$, then there are at least $k+3$ $(k+1)$-kings.
\end{thm}

\begin{proof}
The proofs of the first two statements are straightforward.

For the third, if $n \ge k+2$, Corollary \ref{even-k+1-2-fin} gives us to possibilities.   That every vertex in $C$ is a $(k+1)$-king, and we are done.   Or the existence of $u_1, u_2, u_3 \in V(C)$ such that $u_1$ is a $(k+1)$-king of $D$, $u_2$ is a $(k+2)$-king of $D$, $u_2 \to u_1$, $d(u_2, u_3) = k+2$ and $u_3$ is a $2$-king in $D$.   Also, every vertex at distance $k+2$ from $u_2$ is a $3$-king.

For the second case, let $P = (u_2 = v_0, v_1, \dots, v_{k+2} = u_3)$ a $u_2 u_3$-path.   Since $u_3 = v_{k+2}$ is a $2$-king, then for every vertex $x \in V(D)$ such that $d(x, v_{k+2}) = n$, $x$ is an $(n+2)$-king.   Therefore, $v_i$ is a $(k+1)$-king for every $3 \le i \le k+2$.   Also, Corollary \ref{even-k+1-2-fin} gives us the existence of $u_1$ such that $u_2 = v_0 \to u_1$ and $u_1$ is a $(k+1)$-king.   Since $v_0 \to u_1$, we have that $u_1 \ne v_i$ for every $2 \le i \le k+2$.   Hence, we have found $k+1$ different $(k+1)$-kings.

If $v_2$ is a $(k+1)$-king, we are done.   Otherwise, $v_2$ is a $(k+2)$-king.   Corollary \ref{2k-even} implies the existence of $w_{k+2} \in V(C)$ such that $d(v_2, w_{k+2}) = k+2$ and $w_{k+2}$ is a $2$-king.   Let $P' = (v_2 = w_0, w_1, \dots, w_{k+1}, w_{k+2})$ be a $v_2 w_{k+2}$-path.   Since $w_{k+1} \to w_{k+2}$ and $w_{k+2}$ is a $2$-king, we can conclude that $w_{k+1}$ is a $3$-king, \emph{i.e.}, $w_{k+1}$ and $w_{k+2}$ are $(k+1)$-kings, because $k \ge 4$.   Let us prove that $w_{k+1}$ and $w_{k+2}$ are $(k+1)$-kings different from the previously found $(k+1)$-kings.   Let us recall that $d(v_2, v_j) \le k$ for every $3 \le j \le k+2$ and $d(v_2, w_{k+2}) > d(v_2, w_{k+1}) > k$.   Hence, $w_{k+1}$ and $w_{k+2}$ are different from the vertices $v_j$ with $3 \le j \le k+2$.   And, since $P = (u_2 = v_0, v_1, \dots, v_{k+2} = u_3)$ is a $u_2 u_3$-path of minimum length, the $k$-quasi-transitivity of $D$ implies that $v_k \to u_2$, and hence, $(v_2 P v_k) \cup (v_k, u_2, u_1)$ is a $v_2 u_1$-walk of length $k$, in other words, $d(v_2, u_1) \le k$.   Thus, $w_{k+1}$ and $w_{k+2}$ are different from $u_1$.   In this case we have found $k+3$ different $(k+1)$-kings.

For the final statement, again, if $n \ge k+3$ and every vertex of $C$ is a $(k+1)$-king, there is nothing to prove.   Otherwise, and assuming that $v_2$ is not a $(k+1)$-king, we have already proved the existence of at least $k+3$ $(k+1)$-kings.   If $v_2$ is a $(k+1)$-king, Lemma \ref{dom-k+2-even} implies that $v_{k+2} \to u_1$, hence $(v_2, v_3, \dots, v_{k+2}, u_1)$ is a path of length $k+1$ such that all its $k+2$ vertices are $(k+1)$-kings of $D$.
\end{proof}

\begin{thm}
Let $k \ge 5$ be an odd integer and $D$ a $k$-quasi-transitive digraph with a unique initial strong component $C$ such that $|V(C)|=n$, then:
\begin{enumerate}
	\item If $n \le k$, then the exact number of $(k-1)$-kings of $D$ is $n$.
	
	\item If $n = k+1$, then the exact number of $k$-kings of $D$ is $k+1$.
	
	\item If $n \ge k+2$, then the number of $(k+1)$-kings of $D$ is at least $k+2$.
\end{enumerate}

\noindent
Also, in the last case, the $(k+1)$-kings are distributed on a single path of length $k+1$ or, if $n \ge 3$, then there are at least $k+3$ $(k+1)$-kings.
\end{thm}

\begin{proof}
Again, the proofs of the first two statements are straightforward.

For the third, if $n \ge k+2$, Corollary \ref{odd-k+1-4-fin} gives us to possibilities.   That every vertex in $C$ is a $(k+1)$-king, and we are done.   Or the existence of $u_1, u_2, u_3 \in V(C)$ such that $u_1$ is a $(k+1)$-king of $D$, $u_2$ is a $(k+2)$-king of $D$, $u_2 \to u_1$, $d(u_2, u_3) = k+2$ and $u_3$ is a $4$-king in $D$.   Also, every vertex at distance $k+2$ from $u_2$ is a $4$-king.

For the second case, let $P = (u_2 = v_0, v_1, \dots, v_{k+2} = u_3)$ a $u_2 u_3$-path.   Since $u_3 = v_{k+2}$ is a $4$-king, then for every vertex $x \in V(D)$ such that $d(x, v_{k+2}) = n$, $x$ is an $(n+4)$-king.   Therefore, $v_i$ is a $(k+6-i)$-king, and thus, a $(k+1)$-king for every $5 \le i \le k+2$.   Also, Corollary \ref{odd-k+1-4-fin} gives us the existence of $u_1$ such that $u_2 = v_0 \to u_1$ and $u_1$ is a $(k+1)$-king.   Since $v_0 \to u_1$, we have that $u_1 \ne v_i$ for every $2 \le i \le k+2$.   Hence, we have found $k-1$ different $(k+1)$-kings.

If $v_2, v_3$ and $v_4$ are $(k+1)$-kings, then we have found the $(k+1)$-kings we were looking for.  So we should analyze three cases: $v_4$ is not a $(k+1)$-king; $v_4$ is a $(k+1)$-king but $v_3$ is not; $v_4$ and $v_3$ are $(k+1)$-kings but $v_2$ is not.

We will deal with the first case, the other two cases are very similar.   If $v_4$ is not a $(k+1)$-king, then it is a $(k+2)$-king, because $v_4 \to v_5$ and $v_5$ is a $(k+1)$-king.   Using Corollary \ref{odd-k+1-4-fin} we can consider $w_{k+2} \in V(C)$ such that $d(v_4, w_{k+2}) = k+2$ and $w_{k+2}$ is a $4$-king.   Let $P' = (v_4 = w_0, w_1, \dots, w_{k+1}, w_{k+2})$ be a $v_4 w_{k+2}$-path.   By the choice of $P'$ and since $w_{k+2}$ is a $4$-king, we deduce that $w_{k+1}$ is a $5$-king and $w_k$ is a $6$-king.   But $k \ge 5$, so, $w_k, w_{k+1}$ and $w_{k+2}$ are $(k+1)$-kings because $k \ge 4$.   Besides, since $P = (u_2 = v_0, v_1, \dots, v_{k+2} = u_3)$ is a path of minimum length, we have that $d(v_4, v_j) \le k-2$ for every $5 \le j \le k+2$ and $d(v_4, w_{k+2}) > d(v_4, w_{k+1}) > d(v_4, w_k) > k-1$.   Thus, the vertices $w_k, w_{k+1}$ and $w_{k+2}$ are different from the vertices $v_j$ with $5 \le j \le k+2$.   Also, $P = (u_2 = v_0, v_1, \dots, v_{k+2} = u_3)$ is a path realizing the distance from $u_2$ to $u_3$, hence $v_k \to u_2$ and $(v_4 P v_k) \cup (v_k, u_2, u_1)$ is a $v_4 v_1$-walk of length $k-2$, \emph{i.e.}, $d(v_4, u_1) \le k-2$.   Hence, $w_k, w_{k+1}$ and $w_{k+2}$ are also different from $u_1$ and we have found the $k+2$ different $(k+1)$-kings.
\end{proof}

We will finalize this section improving the existing results for the number of $3$-kings in quasi-transitive digraphs.   These results are implicit in \cite{BJQT}.   We need an additional definition and two results.   Let $D$ be a digraph with vertex set $V(D) = \{ v_1, \dots, v_n \}$ and $H_1, \dots, H_n$ a family of vertex disjoint digraphs.   The \emph{composition} $D[H_1, \dots, H_n]$ of digraphs $D$ and $H_1, \dots, H_n$ is the digraph having  $\bigcup_{i=1}^n V(H_i)$ as its vertex set and arc set $\bigcup_{i=1}^n A(H_i) \cup \left\{ (u,v) \colon\ u \in V(H_i), v \in V(H_j), (v_i, v_j) \in A(D) \right\}$.

In \cite{BJQT}, the following three results are proved.

\begin{thm} \label{qt-chs}
Let $D$ be a quasi-transitive strong digraph.    Then there exist a natural number $q \ge 2$, a strong semicomplete digraph $Q$ on $n$ vertices and quasi-transitive digraphs $W_1, \dots,  W_q$, where each $W_i$ is either a single vertex or a non-strong quasi-transitive digraph, such that $D= Q[W_1, \dots, W_q]$. Furthermore, if $Q$ has a cycle of length two induced by vertices $v_i$ and $v_j$, then the corresponding digraphs $W_i$ and $W_j$ are trivial, that is, each of them has only one vertex.
\end{thm}

\begin{lem} \label{3kqtlem}
	Let $D = Q[W_1, \dots, W_q]$ be a strong quasi-transitive digraph.
	\begin{enumerate}
	\item A vertex $u \in W_i$ is a $3$-king if and only if the vertex $v_i$ of $Q$ is a $3$-king in $Q$.
	\item If $v_i$ is a $3$-king of $Q$ and $W_i$ is non-trivial, then $W_i$ contains a vertex which is a $3$-king but is not a $2$-king.
	\end{enumerate}
\end{lem}

\begin{thm} \label{2kqt}
If $D = Q[W_1, \dots, W_q]$ is a strong quasi-transitive digraph, then D has a $2$-king if and only if $|W_i| = 1$ for some $i$ such that $v_i$ is a 2-king in $Q$.
\end{thm}

The following results cover the case for $k=2$, which was missing from Theorem \ref{k+2-k+1}.

\begin{thm}
	Let $D$ be a quasi-transitive digraph with a unique initial component $C$.   If there are at least four vertices in $C$, then there are at least four $3$-kings in $D$.   If there are three vertices in $C$, then every vertex of $C$ is a $2$-king and if there are one or two vertices in $C$, every vertex of $C$ is a $1$-king.
\end{thm}

\begin{proof}
	 In virtue of Lemma \ref{k-1-dom-comp}, it suffices to prove the result for strong digraphs.   Let $Q$ be a strong semicomplete digraph with vertex set $\{ v_1, \dots, v_q \}$.   If $2 \le q \le 4$, then every vertex of $Q$ is a $3$-king.   If $q \ge 5$, a well known result tells us that there are at least three $2$-kings, and being $Q$ strong, then there must be at least one other vertex dominating one of these $2$-kings, hence, there are at least four $3$-kings.   The existence of four $3$-kings in a quasi-transitive strong digraph with at least four vertices follows from Lemma \ref{3kqtlem}.   The remaining cases are trivial.
\end{proof}

\begin{thm}
	Let $D$ be a quasi-transitive digraph with a unique initial component $C$.   If there are no $2$-kings in $D$, then there are at least seven $3$-kings.
\end{thm}

\begin{proof}
	Again, it suffices to prove the result for strong digraphs, so let $D = Q[W_1, \dots, W_q]$ be a quasi-transitive strong digraph.   From the previous theorem we have that there are at least four $3$-kings $x_1, x_2, x_3, x_4$ in $D$.   Moreover, we can suppose without loss of generality that $x_i$ is a vertex of $W_i$ and $v_i$ is a $2$-king of $Q$ for $1 \le i \le 3$.   Since there are not $2$-kings in $D$, Theorem \ref{2kqt} impliest that $|W_i| \ge 2$ for $1 \le i \le 3$.   The existence of the additional three $3$-kings follows from Lemma \ref{3kqtlem}.
\end{proof}

\section{$(k+2)$-kernels in $k$-quasi-transitive digraphs} \label{kernels}

Let $D= (V, A)$ be a digraph.   A set $S \subseteq V$ is $k$-independent if for every $u,v \in S$, $d(u,v) \ge k$; and it is $l$-absorbent if for every $u \in V \setminus S$, there exists $v \in S$ such that $d(u,v) \le l$.   We say that $K \subseteq V$ is a $(k,l)$-kernel of $D$ if it is $k$-independent and $l$-absorbent.   A $k$-kernel is a $(k,k-1)$-kernel.   In \cite{GS-HC}, it is proved that for every even integer $k \ge 2$, a $k$-quasi-transitive digraph always has a $(k+2)$-kernel.   Also, it is conjectured that an analogous result is valid for odd integers.   Here, we prove that conjecture to be true.

\begin{thm}
Let $k \ge 3$ be an odd integer.   If $D$ is a $k$-quasi-transitive digraph, then $D$ has a $(k+2)$-kernel.
\end{thm}

\begin{proof}
Observe that $D$ is a $k$-quasi-transitive digraph if and only if $\overleftarrow{D}$ is a $k$-quasi-transitive digraph.   Hence, we can consider the initial components $I_1, I_2, \dots, I_r$ of $\overleftarrow{D}$.    Lemma \ref{exis-deg-odd} gives us the existence of a $(k+1)$-king $v_j \in V(I_j)$.   Let us recall that every vetex in $D$ must be reached from an initial strong component.   Hence, it can be obtained by Lemma \ref{k-1-dom-comp} that every vertex in $\overleftarrow{D}$ must be reached at distance less than or equal to $k+1$ by a vertex in $\{ v_1, v_2, \dots, v_r \} = K$.   Since each $v_j$ is in a distinct initial component, it is clear that $K$ is $(k+2)$-independent.   Thus, $K$ is a $(k+2)$-independent and $(k+1)$-absorbent set in $D$, \emph{i.e.}, a $(k+2)$-kernel.
\end{proof}

We can observe that this result is as good as it gets in terms of $k$-kings.   Let $k \ge 2$ be an integer, and $D$ the digraph consisting of the path $(v_0, v_1, \dots, v_{k+1})$, of length $k+1$,  with two additional arcs, $(v_k, v_0)$ and $(v_{k+1}, v_1)$.   It is direct to observe that $D$ is a strong $k$-quasi-transitive digraph without a $k$-king, nonetheless, $\{ v_0, v_{k+1} \}$ is a $(k+1)$-kernel of $D$.   It was proved in \cite{GS-HC2} that every $2$-quasi-transitive digraph has a $3$-kernel and in \cite{GS-HC} that every $3$-quasi-transitive digraph has a $4$-kernel.   So the following conjecture also stated in \cite{GS-HC} seems reasonable.

\begin{conjecture}
Let $k \ge 2$ be an integer.   Every $k$-quasi-transitive digraph has a $(k+1)$-kernel.
\end{conjecture}

If true, this would be sharp:   A $(k+1)$-cycle is a $k$-quasi-transitive digraph without a $k$-kernel.


\end{document}